\documentclass[10pt]{article}
\usepackage[utf8x]{inputenc}
\usepackage[T1]{fontenc}
\usepackage{fullpage}
\usepackage {amssymb}
\usepackage {amsmath}
\usepackage {amsfonts}
\usepackage{algorithm}
\usepackage {latexsym}
\usepackage {paralist}
\usepackage{listings}
\usepackage{multirow}
\usepackage{subfig}
\usepackage{tikz}
\usepackage{pgfplots}
\pgfplotsset{width=7cm,
        standard/.style={
        axis x line=middle,
        axis y line=middle,
        enlarge x limits=0.15,
        enlarge y limits=0.15,
        every axis x label/.style={at={(current axis.right of origin)},anchor=north west},
        every axis y label/.style={at={(current axis.above origin)},anchor=north east}
    }
}
\usepackage{xcolor}
\usepackage{url}
\allowdisplaybreaks
 \newcommand\ForAuthors[1]
 {\par\smallskip                     
  \begin{center}
   \fbox
   {\parbox{0.9\linewidth}
    {\raggedright\sc--- #1}
   }
  \end{center}
  \par\smallskip                     %
 }
\usepackage{color}
\newcommand{\assale}[1]{{\color{blue} #1}}
\newcommand{\comment}[1]{}
\def\norm#1{\mbox{$\| #1 \|$}}
\def\ineq{\leq_{w,s}}
\def\nn{\mathbb N}

\newcommand{\rr}{\mathbb R}
\newcommand{\rd}{\rr^d}

\newcommand{\pp}{\mathbb P}

\newcommand{\Mat}{\mathbf{M}}
\newcommand{\Id}{\operatorname{Id}}

\newcommand{\Sdp}[1]{\mathbb{S}_{#1}^{+}}

\newcommand{\Sp}[1]{\mathbb{S}_{#1}^{\geq 0}}
\newcommand{\Sy}[1]{\mathbb{S}_{#1}}
\newcommand{\mt}[2]{\mathbb{M}_{#1 \times #2}}

\newcommand{\ccop}[2]{\mathbf{C}_{#1}\left(#2\right)}
\newcommand{\cop}[1]{\mathbf{C}_{#1}}
\newcommand{\psd}[1]{#1_{+}}
\newcommand{\pos}[1]{#1_{p}}

\newcommand{\bHom}[1]{\overline{\mathbf{H}\left(#1\right)}}
\newenvironment{psmallmatrix}
  {\left(\begin{smallmatrix}}
  {\end{smallmatrix}\right)}

\def\st{\operatorname{s.t.}}
\def\FS{\mathcal{S}}
\def\Sol#1{\operatorname{Sol}_{\lambda}\left(#1\right)}
\def\ux{\begin{pmatrix} 1\\ x\end{pmatrix}}

\def\betab{\omega}

\newcommand{\rel}[1]{#1^{\mathcal R}}
\newcommand{\sha}[1]{#1^{\sharp}}

\newcommand{\Sw}{\mathrm{Sw}}
\newcommand{\In}{\mathrm{In}}
\newcommand{\rea}{\mathcal{R}}

\newcommand{\state}{\mathcal{X}}
\newcommand{\laws}{\mathcal{A}}
\newcommand{\map}[1]{\mathbb{A}^{#1}}

\newcommand{\pws}{P}
\newcommand{\mybrackets}[1]{\big(#1\big)}

\newcommand{\ind}{\mathcal{I}}

\newcommand{\Pq}{\mathcal{P}}
\newtheorem{definition}{Definition}
\newtheorem{lemma}{Lemma}
\newtheorem{corollary}{Corollary}
\newtheorem{proposition}{Proposition}
\newtheorem{theorem}{Theorem}
\newtheorem{example}{Example}
\newtheorem{proof}{Proof}
\title{Overapproximating the Reachable Values Set of Piecewise Affine Systems Coupling Policy Iterations with Piecewise Quadratic Lyapunov
Functions}
\author{Assal\'e Adj\'e\thanks{The author has been supported by the CIMI (Centre International de Math\'ematiques et d'Informatique) Excellence program ANR-11-LABX-0040-CIMI within the program
ANR-11-IDEX-0002-02 during a postdoctoral fellowship.}\\assale.adje@irit.fr\\Institut de Recherche en Informatique de Toulouse (IRIT)\\ Universit\'e Paul Sabatier\\ Toulouse, France}
\date{}
\begin{document}
\maketitle

\begin{abstract}
We have recently constructed a piecewise quadratic Lyapunov function to prove the boundedness of the reachable values set of piecewise affine discrete-time systems. The method developed also provided an overapproximation of the reachable values set. In this paper, we refine the latter overapproximation extending previous works combining policy iterations with quadratic Lyapunov functions.  
\end{abstract}



\section{Introduction}
Several catastrophic events showed the importance of the formal verification of programs. Some of these failures
are caused by overflows. A method to prove the absence of overflows in numerical programs consists in providing 
precise safe bounds over the reachable states of the program variables. 

In this paper, we are interesting in a particular class of numerical programs: single while loop programs with a switch-case structure
inside the loop body. Moreover, we suppose that test and assignment functions are affine. These programs can be represented as 
piecewise affine discrete-time systems. To overapproximate the reachable states of the program variables is thus reduced to 
overapproximate the reachable values set of a piecewise affine discrete-time system. Hence, we propose to compute \emph{automatically} 
precise bounds over piecewise affine discrete-time systems using policy iterations and piecewise quadratic Lyapunov functions.   

Initially policy iteration solves stochastic control problems~\cite{howard:dp} which can be reduced to solve fixed point problems 
involving functions with maxima of affine functions coordinates. Policy iteration was then extended to zero-sum two-player stochastic 
games~\cite{doi:10.1287/mnsc.12.5.359}, this extension allows the computation of the unique fixed point of min-max of affine maps.
The very first extension of policy iterations in program analysis was in 2005 by Costan et al~\cite{costan2005policy}. 
Since then, the usage of policy iteration in various verification problems greatly increases: in~\cite{DBLP:journals/jsc/GawlitzaSAGG12},
the authors describe policy iteration algorithm to overapproximate the reachable values set of numerical programs with affine assignments; 
in~\cite{Masse201277}, the author proves termination by policy iteration; in~\cite{Schrammel,sotin2011policy} the authors propose to embed policy iterations for
programs dealing with both numerical and boolean variables. 

The method developed in~\cite{adje2014piecewise} allows to prove that the reachable values set of a piecewise affine system is bounded. 
The method relies on the synthesis of a piecewise quadratic Lyapunov function of this piecewise affine system. The problem formulation makes appear as a decision variable an upper bound on the Euclidian norm of the state variable. This upper bound can be very loose since it combines all the coordinates together. We propose to use a templates based method. A templates method consists in representing 
sets as sublevel sets of \emph{given} functions called \emph{templates}. Then to compute an overapproximation is reduced to computing bounds over the templates. The most precise overapproximation with respect to these templates is provided by the vector of bounds satisfying a smallest fixed point.   
In our context, the generated piecewise quadratic Lyapunov function is used as a template. We complete the templates basis by the square of variables. Finally we use policy iterations to solve the (smallest) fixed point equation. Thus, policy iterations algorithm leads to tighter bounds over the reachable values set.

The use of quadratic Lyapunov function as quadratic templates was explicitly done in~\cite{DBLP:conf/hybrid/RouxJGF12} but it is not enough to prove the boundedness of reachable values set of a piecewise affine system unless that a common quadratic Lyapunov function exists. 
Policy iteration algorithms in templates domain proposed in~\cite{dblp1968095,DBLP:journals/jsc/GawlitzaSAGG12} used quadratic templates and did not handle piecewise quadratic templates. In this paper, we adapt policy iteration based on Lagrange duality~\cite{adjepolicy} to piecewise quadratic functions. The works on piecewise quadratic Lyapunov functions~\cite{1205102,912814} are also related to this paper. Their authors are interested in proving stability of piecewise linear systems. However, as classical quadratic Lyapunov functions, piecewise quadratic Lyapunov functions provide sublevel invariant sets to the system. We use this latter interpretation for a verification purpose. Finally, note that tropical polyhedra domain~\cite{AllamigeonThesis} generates disjunctions of zones as invariants. The latter invariants did not encode quadratic relations between variables. 
  
The first contribution of the paper is the formalisation of piecewise quadratic Lyapunov functions to prove the boundedness of the trajectories of
a piecewise affine discrete-time dynamical system. This formalisation uses the theory of cone-copositive matrices which is also an 
original contribution in this context. 

The main contribution of the article is the extension of policy iterations algorithm to the piecewise quadratic Lyapunov functions in order to provide precise bounds on the reachable values. Indeed, policy iteration has just been constructed in the case of quadratic functions.         


\subsubsection*{Notations}
Numbers. $\nn$ denotes the set of nonnegative integers, then for $d\in\nn$,
$[d]=\{1,\ldots,d\}$. $\rr$ is the set of reals, $\rr_+$ the set of nonnegative reals and $\rr^d$ denotes the set of vectors of $d$ reals.
We denote by $\wp(\rd)$ the set of subsets of $\rd$. 

Inequalities. For $y,z\in\rd$, $y<z$ (resp. $y\leq z$) means $\forall\, l\in [d]$, $y_l<z_l$, (resp. $\forall\, l\in [d]$, $y_l\leq z_l$)
and $y\ineq z$ is a mix of weak and strict inequalities.

Matrices. $\mt{n}{m}$ is the set of matrices with $n$ rows and $m$ columns. $0_{n,m}$ and $0_n$ are respectively the null matrices of $\mt{n}{m}$ and $\mt{n}{n}$.
$\Id_{n}$ is the identity matrix of $\mt{n}{n}$. $M^\intercal$ is the transpose of $M\in\mt{n}{m}$. $\Sy{n}$ is the set of symmetric matrices of size $n\times n$. 
$A\succeq 0$ means that $A$ is semi-definite positive i.e. $A\in\Sy{d}$ and $\forall\, x\in\rd$, $x^\intercal A x\geq 0$. $\Sdp{d}$ is the convex cone
of semidefinite positive matrices.

\section{Piecewise affine discrete-time systems}
\label{statement}
In this section, we detail the systems we will consider in the paper.

Piecewise affine systems (PWA for short) are defined as systems the dynamic of which is piecewise affine and thus the dynamic is characterized by 
a polyhedral partition and a family of affine maps relative to this partition. For us, a polyhedral partition is a family of convex 
polyhedra such that:
\begin{equation}
\label{partition}
\bigcup_{i\in \ind}X^i=\rd\text{ and }\forall\, i,j\in \ind,\ i\neq j\ X^i\cap X^j=\emptyset\enspace . 
\end{equation}
The convex polyhedron $X^i$ can contain both strict and weak inequalities and is represented by $T^i\in\mt{n_i}{m}$ 
and $c^i\in\rr^{n_i}$. We denote by $T_s^i$ (resp. $T_w^i$) and $c_s^i$ (resp. $c_w^i$) the parts of $T^i$ and $c^i$ 
corresponding to strict (resp. weak) inequalities: 
\begin{equation}
\label{polyhedra}
\begin{array}{ll}
X^i&=\left\{x\in\rd \left| T^i x \ineq c^i\right\}\right.\\
   &=\left\{x\in\rd \left| T_s^i x < c_s^i,\ T_w^i x \leq c_w^i\right\}\right.
\end{array}
\end{equation}
\begin{definition}[Piecewise Affine System]
\label{pwadef}
A PWA is characterized by the triple $(X^0,\state,\laws)$ where:
\begin{itemize}
\item $X^0$ is the polytope of the initial conditions of the form~\eqref{polyhedra};
\item $\state:=\{X^i, i\in \ind\}$ is a polyhedral partition i.e. satisfying~\eqref{partition};
\item $\laws:=\{x\mapsto f^i(x)=A^i x+b^i, i\in\ind\}$ where $A^i\in\mt{d}{d}$ and $b^i\in\rd$;
\end{itemize}
And satisfies the following relation for all $k\in\nn$:
\begin{equation}
\label{pwa}
x_0\in X^0,\ \text{if } x_k\in X^i,\ x_{k+1}=f^i(x_k)\enspace.
\end{equation}
\end{definition} 
Let $\pws=(X^0,\state,\laws)$ be a PWA. We now define some tools that we need during the analysis. 
First we define the reachable values set $\rea$ of $\pws$:
\begin{equation}
\label{reach}
\rea=\bigcup_{k\in\nn} {\map{}}^k(X^0),\text{ where } \map{}(x)=f^i(x) \text{ if } x\in X^i
\end{equation}
We define the set of possible switches:
\begin{equation}
\label{switches}
\begin{array}{l}
\Sw:=\{(i,j)\in \ind^2\mid \rea\cap X^{ij}\neq \emptyset\}\\ 
\text{ where }X^{ij}=X^i\cap {f^{i}}^{-1}(X^j)\enspace.
\end{array}
\end{equation}
Finally, we define the set of indices of polyhedra of $\state$ which meet the polyhedron of possible initial conditions:
\begin{equation}
\label{initialpol}
\In:=\{i\in \ind\mid X^{i0}\neq \emptyset\} \text{ where } X^{i0}=X^i\cap X^0\enspace.
\end{equation} 
We introduce for $i\in\ind$, the following matrix of $\mt{(d+1)}{(d+1)}$:
\begin{equation} 
\label{homogeneouslaw}
F^i=\begin{pmatrix} 
    1 & 0_{1\times d} \\
    b^i & A^i 
    \end{pmatrix}\enspace .
\end{equation}
Eq.~\eqref{pwa} can be rewritten as  $(1,x_{k+1})^\intercal =F^i(1,x_k)$.

We are interested in computing \emph{automatically} precise overapproximation of $\rea$. We propose to compute 
an overapproximation of $\rea$ as a set $S\subseteq \rd$ such that $X^0\subseteq S$ and
$\forall\, i\in\ind,\ x\in S\cap X^i \implies A^i x+b^i\in S$.
The set $S$ can be computed as a sublevel of a Lyapunov function containing the initial 
states. 

From now, we work with a fixed PWA $\pws=(X^0,\state,\laws)$, where $X^0$, $\state$ and $\laws$ are of the form
of Def.~\ref{pwadef}. 

\section{Piecewise quadratic Lyapunov functions}
\label{piecewiselyapunov}
In this paper, we use piececewise quadratic Lyapunov functions for piecewise affine systems to compute directly 
an overapproximation of reachable values set.

Let $q$ be a quadratic form i.e. a function such that for all $y\in\rd$, $q(y)=y^\intercal A_q y+b_q^\intercal y+c_q$ where $A_q\in\Sy{d}$, $b_q\in\rd$ and $c_q\in\rr$.
We define the lift-matrix of $q$, the matrix of $\Sy{d+1}$ defined as follows:
\begin{equation}
\label{liftmatrix}
\Mat(A_q,b_q,c_q)=\Mat(q)=\begin{pmatrix} 
c_q & (b_q/2)^\intercal \\ 
(b_q/2) & A_q 
\end{pmatrix}
\end{equation}
It is obvious that the $q\mapsto \Mat(q)$ is linear. Let $A\in\mt{d}{d}$, $b\in\rd$, and $q$ be a quadratic form, we have, for all $x\in\rd$: 
\begin{equation}
\label{lemmasimple}
q(Ax+b)=\ux^\intercal \begin{pmatrix}
1& 0_{1\times d}\\
b & A
\end{pmatrix}^\intercal \Mat(q) \begin{pmatrix}
1& 0_{1\times d}\\
b & A
\end{pmatrix} \ux\enspace .
\end{equation}
\begin{lemma}
\label{quadextension}
Let $A\in\Sy{d}$, $b\in\rd$ and $c\in \rr$. Then: $(\forall\, y\in\rd,\ y^\intercal A y+b^\intercal y+c\geq 0)\iff \Mat(A,b,c)\in \Sdp{d+1}$
\end{lemma} 
\begin{definition}[(Cone)-copositive matrices]
Let $M\in\mt{m}{d}$. A matrix $Q\in\Sy{d}$ which satisfies 
\[
My\geq 0\implies y^\intercal Q y\geq 0
\] 
is called $M$-copositive.

An $\Id_{d}$-copositive matrix is called a \emph{copositive matrix}. We denote by $\ccop{d}{M}$ the set of $M$-copositive matrices and $\cop{d}$ the set of copositive matrices.
\end{definition} 
For $P\in\mt{n}{m}$ and $c\in\rr^n$, we define the following matrix:
\begin{equation}
\label{homogeneouspol}
\bHom{P,c}=
\begin{pmatrix}
1 &  & 0_{1\times m}\\
c &  &-P
\end{pmatrix}
\in\mt{(n+1)}{(m+1)}
\end{equation}
\begin{lemma}
Let $P\in\mt{n}{m}$ and $c\in\rr^n$. Then, for all $x\in\rr^n$, $Px\leq c\iff \bHom{P,c}\ux \geq 0$.
\end{lemma}
\begin{lemma}
\label{polyhedralcopo}
Let $q:\rd\to\rd$ be a quadratic function. Let $M\in\mt{m}{d}$ and $p\in\rr^m$. Let us consider $C=\{x\mid Mx\leq p\}$. Then $\Mat(q)\in \ccop{d+1}{\bHom{M,p}}\implies (q(x)\geq 0,\ \forall\, x \in C)$.
\end{lemma}
and we introduce the following matrices: 
\begin{subequations}
\begin{gather}
\forall\, i\in \ind,\ E^i=\bHom{T^i,c^i}\enspace ,\\
\forall\, (i,j)\in\ind^2,\ E^{ij}=\bHom{\begin{pmatrix}T^i\\ T^jA^i\end{pmatrix},\begin{pmatrix}c^i\\ c^j-T^jb^i\end{pmatrix}}\enspace ,\\
\forall\, i\in\In,\ E^{i0}=\bHom{\begin{pmatrix}T^i\\ T^0\end{pmatrix},\begin{pmatrix}c^i\\ c^0\end{pmatrix}}\enspace.
\end{gather}
\end{subequations}
\begin{lemma}
\label{overpol}
For all $i\in\ind$,
$X^i\subseteq \{x\mid E^i (1\ x^\intercal)^\intercal \geq 0\}$, for all $(i,j)\in\Sw$, $X^{ij}\subseteq \{x\mid E^{ij} (1\ x^\intercal)^\intercal \geq 0\}$
and for all $i\in\In$, $X^{i0}\subseteq \{x\mid E^{i0} (1\ x^\intercal)^\intercal \geq 0\}$.
\end{lemma}
\begin{definition}[PQL functions]
\label{piecewisequad}
A function $L$ is a piecewise quadratic Lyapunov function (PQL for short) for $\pws$ if and only if there exist a family $\{(P^i,q^i),P^i\in\Sy{d},q^i\in\rd,\ i\in \ind\}$ and two reals $\alpha$ and $\beta$ such that:
\begin{enumerate}
\item $\forall\, i\in \ind$, $\forall\, x\in X^i$, $L(x)=L^i(x)=x^\intercal P^i x+2x^\intercal q^i$;
\item $\forall\, i\in \ind$:
\begin{equation}
\label{boundedeq}
\Mat(P^i,2q^i,-\alpha)-\Mat(\Id,0,-\beta)\in\ccop{d+1}{E^i}\enspace ;
\end{equation}
\item $\forall\, (i,j)\in \Sw$:
\begin{equation}
\label{stabilityeq}
\Mat(P^i,2q^i,0)-{F^i}^\intercal \Mat(P^j,2q^j,0) F^i\in\ccop{d+1}{E^{ij}}\enspace ;
\end{equation}
\item $\forall\, i\in \In$:
\begin{equation}
\label{initialeq}
-\Mat(P^i,2q^i,-\alpha)\in\ccop{d+1}{E^{i0}}\enspace .
\end{equation}
\end{enumerate}
\end{definition}
\begin{theorem}[Bounded trajectories]
\label{boundedth}
Assume that $\pws$ admits a PQL function characterized by $\{(P^i,q^i),P^i\in\Sy{d},q^i\in\rd,\ i\in \ind\}$ and reals $\alpha$ and $\beta$.
Let $i\in\ind$, $S_{\alpha}^i=\{x\in X^i\mid L^i(x)\leq \alpha\}=\{x\in X^i\mid x^\intercal P^i x+2x^\intercal q^i\leq \alpha\}$ and 
$S=\cup_{i\in \ind} S_{\alpha}^i$. Then, $\rea\subseteq S\subseteq \{x\in\rd\mid \norm{x}_2^2\leq \beta\}$.
\end{theorem}
\begin{proof}
First, we prove that $S\subseteq \{x\in\rd\mid \norm{x}_2^2\leq \beta\}$. Let $i\in \ind$ and $x\in X^i$. From Eq.~\eqref{boundedeq}, Lemma~\ref{polyhedralcopo} and Lemma~\ref{overpol}, 
$x^\intercal P^i x+2x^\intercal q^i-\alpha-\norm{x}_2^2+\beta\geq 0$. This is equivalent to $\beta-\norm{x}_2^2\geq \alpha-x^\intercal P^i x-2x^\intercal q^i$ which implies that 
$S\subseteq \{x\in\rd\mid \norm{x}_2^2\leq \beta\}$.

Now, we have to prove $\rea\subseteq S$. From Eq.~\eqref{reach}, we have to prove that for all $k\in \nn$, $\map{k}(X^0)\subseteq S$. We prove it by induction on $k$. 
Let $x\in X^0$. Since $\state$ satisfies~\eqref{partition}, there exists a unique $i\in\In$ such that $x_0\in X^{i0}$. From Eq.~\eqref{initialeq}, Lemma~\ref{polyhedralcopo} and Lemma~\ref{overpol},
$L^i(x)\leq \alpha$. Now suppose $\map{k}(X^0)\subseteq S$ for some $k\in\nn$. Let $y\in\map{k+1}(X^0)$. Then $y=\map{}(x)$ for some $x\in\map{k}(X^0)$. 
Since $\state$ satisfies~\eqref{partition}, there exists an unique $(i,j)\in\Sw$ such that $x\in X^{ij}$ (hence $y \in X^j$). As $x\in X^i$ and $x\in S$, then $x\in S_\alpha^i$.
From Eq.~\eqref{stabilityeq}, Lemma~\ref{polyhedralcopo} and Lemma~\ref{overpol}, $0\leq L^i(x)-L^j(y)=L^i(x)-\alpha-(L^j(y)-\alpha)$. As $x\in S_\alpha^i$, $0\geq L^i(x)-\alpha$ which implies that $0\geq L^j(y)-\alpha$ and finally $y\in S_\alpha^j\subseteq S$.     
\end{proof}    
\subsection{Computational issues}
To construct PQL functions, we are faced with two issues. First, we must know the sets of indices $\Sw$ and $\In$.
Second we have to manipulate cone-copositive constraints. 
\subsubsection{The computation of sets $\Sw$ and $\In$}
To set $\Sw$ is defined from $\rea$, the set which we want approximate. To overcome this issue, we consider a bigger set 
by removing the intersection with $\rea$:
\begin{equation}
\label{newswitches}
\overline{\Sw}:=\{(i,j)\in\ind^2\mid X^{ij}\neq \emptyset\}\enspace .
\end{equation} 
Since $X^i$ and $X^j$ can contain strict inequalities, we can use alternative theorems such as Motzkin's theorem~\cite{Motzkin}
to compute $\overline{\Sw}$. Note that we use this technique based LP to determine exactly $\In$. 

The direct application of Motzkin's transposition theorem~\cite{Motzkin} yields to the next proposition.
\begin{proposition}
\label{propSwitches}
Let $n_{ij}^s$ (resp. $n_{ij}^w$) be the number of strict (resp. weak) inequalities in $X^i\cap X^j$. The couple $(i,j)\in\overline{\Sw}$ if and only if:
\[
\left\{
\begin{array}{l}
\displaystyle{
\begin{pmatrix}
1 & & 0_{1\times d}\\
c_s^i & &-T_s^i\\
c_s^j-T_s^j b^i& &-T_s^jA^i
\end{pmatrix}^\intercal p^s+\begin{pmatrix}
c_w^i& & -T_w^i\\
c_w^j-T_w^j b^i & &-T_w^j
\end{pmatrix}^\intercal p=0}\\
\\
\displaystyle{\sum_{k=1}^{n_{ij}^s+1} p_k^s=1},\ p^s\geq 0,\ p\geq 0
\end{array} 
\right.
\]
has no solution.

Let $n_{i0}^s$ (resp. $n_{i0}^w$) be the number of strict (resp. weak) inequalities in $X^i\cap X^0$. The index $i\in\In$ if and only if:
\[
\left\{
\begin{array}{l}
\begin{pmatrix}
1 & & 0_{1\times d}\\
c_s^i & &-T_s^i\\
c_s^0& &-T_s^0
\end{pmatrix}^\intercal p^s+\begin{pmatrix}
c_w^i& & -T_w^i\\
c_w^0& &-T_w^0
\end{pmatrix}^\intercal p=0\\
\\
\displaystyle{\sum_{k=1}^{n_{i0}^s+1} p_k^s=1,\ p^s\geq 0,\ p\geq 0}
\end{array} 
\right.
\]
has no solution.
\end{proposition}
\subsubsection{Cone-copositive constraints}
Cone-copositive matrix characterizations is an intensive research field and a list of interesting papers about can be found 
in~\cite{SurveyCopositive}. 
\begin{proposition}[Th. 2.1 of~\cite{Martin1981227}]
\label{conecopositive}
Let $M\in\mt{m}{d}$. Then:
\begin{equation}
\label{eqcop}
\tag{$\Delta$}
\{M^\intercal C M+S\mid C\in\cop{d}\text{ and } S\in \Sdp{d}\}\subseteq \ccop{d}{M}
\end{equation}
If the rank of $M$ is equal to $m$, then~\eqref{eqcop} is actually an equality. 
\end{proposition}
The next proposition discusses simple a characterization of copositive matrices as a sum of a semi-definite positive matrix and a nonnegative matrix. 
\begin{proposition}[~\cite{PSP:2051276,PEM:3071212}]
\label{copositivecha}
We have: $\forall\, d\in\nn$: $\Sp{d}+\Sdp{d}\subseteq \cop{d}$. If $d\leq 4$ then $\cop{d}=\Sp{d}+\Sdp{d}$.
\end{proposition}
\begin{corollary}
\label{copositiverestriction}
Let $M\in\mt{m}{d}$. Then:
\begin{equation}
\label{eqcorollary}\tag{$\star$}
\ccop{d}{M}\supseteq\left\{Q\in\Sy{d}\left| \begin{array}{c}\exists\, \pos{W}\in\Sp{m},\ \psd{W} \in \Sdp{m},\
 \st\\ Q-M^\intercal \left(\pos{W}+\psd{W}\right) M\succeq 0\end{array}\right\}\right.
\end{equation}
If $M$ has full row rank and $d\leq 4$, then~\eqref{eqcorollary} is actually an equality.
\end{corollary}
Copositive constraints study is a quite recent field of research. Algorithms
exist (e.g.~\cite{Bundfuss:2009:ALA:1654367.1654377}) but for the knowledge of the author no tools are available. In this paper, in practice, we use Corollary~\ref{copositiverestriction} 
and we replace $\ccop{d}{M}$ by the right-hand side of Eq.~\eqref{eqcorollary}.
\comment{
Nevertheless, by taking $\Sp{m}+\Sdp{m}$, we allow symmetry. A way to break the symmetry is to simply add the trivial row vector with one on the first coefficient and 0 otherwise as we do in Def.~\ref{homogeneouscell}.
} 
\subsubsection{Computation of Piecewise quadratic Lyapunov functions using SDP solvers}
Finally, we construct PQL functions using semidefinite programming. We define the notion of computable 
PQL functions. 
\begin{definition}[Computable PQL functions]
\label{computablepiecewise}
A function $L$ is a computable PQL for to $\pws$ if and only if there exist two reals $\alpha$ and $\beta$ and four families:
 \begin{itemize}
 \item $\Pq:=\{(P^i,q^i),P^i\in\Sy{d},q^i\in\rd,\ i\in \ind\}$
\item $\mathcal{W}:=\{\left(\pos{W^i},\psd{W^i}\right)\in \Sp{n_i+1}\times \Sdp{n_i+1},i\in\ind\}$, 
\item $\mathcal{U}:=\{\left(\pos{U^{ij}},\psd{U^{ij}}\right)\in \Sp{n_{ij}}\times \Sdp{n_{ij}},(i,j)\in\overline{\Sw}\}$
\item $\mathcal{Z}:=\{\left(\pos{Z^{i0}},\psd{Z^{i0}}\right)\in \Sp{n_{i0}}\times \Sdp{n_{i0}},i\in\In\}$
\end{itemize}
such that:
\begin{enumerate}
\item $\forall\, i\in \ind$, $\forall\, x\in X^i$, $L(x)=L^i(x)=x^\intercal P^i x+2x^\intercal q^i$;
\item $\forall\, i\in \ind$:
\begin{equation}
\label{boundedeqrelax}
\begin{array}{rr}
\Mat(P^i,2q^i,-\alpha)-\Mat(\Id,0,-\beta)& \hphantom{ \succeq 0}\\
 \hphantom{\Mat(P^i,2q^i,-\alpha)}-{E^i}^\intercal \left(\pos{W^i}+\psd{W^i}\right) E^i & \succeq 0\enspace ;
\end{array}
\end{equation}
\item $\forall\, (i,j)\in \overline{\Sw}$:
\begin{equation}
\label{stabilityeqrelax}
\begin{array}{rr}
 \Mat(P^i,2q^i,0)-{F^i}^\intercal \Mat(P^j,2q^j,0) F^i & \hphantom{ \succeq 0}\\
  \hphantom{\Mat(P^i,2q^i,0) }-{E^{ij}}^\intercal \left(\pos{U^{ij}}+\psd{U^{ij}}\right) E^{ij}& \succeq 0\enspace ;
 \end{array}
\end{equation}
\item $\forall\, i\in \In$:
\begin{equation}
\label{initialeqrelax}
-\Mat(P^i,2q^i,-\alpha)-{E^{i0}}^\intercal \left(\pos{Z^{0i}}+\psd{Z^{0i}}\right) E^{i0}\succeq 0;
\end{equation}
\end{enumerate}
\end{definition} 
Let us consider the problem:
\begin{equation}
\label{SDPlyap}
\tag{PSD}
\begin{array}{cl}
\displaystyle{\inf_{\substack{\Pq,\mathcal{W},\mathcal{U},\mathcal{Z},\\ \alpha,\beta}}} & \alpha+\beta\\
\st & 
\left\{
\begin{array}{l}
(\Pq,\mathcal{W},\mathcal{U},\mathcal{Z},\alpha,\beta) \text{ satisfies}~\eqref{boundedeqrelax},~\eqref{stabilityeqrelax} \text{ and}
~\eqref{initialeqrelax}\\
\alpha\geq 0,\ \beta\geq 0
\end{array}
\right.
\end{array}
\end{equation}
Problem~\eqref{SDPlyap} is thus a semi-definite program.
The use of the sum $\alpha+\beta$ as objective function enforces the functions $L^i$s to provide a minimal bound $\beta$
and a minimal ellispoid containing the initial conditions. The constraint $\beta\geq 0$ is obvious since $\beta$ represents 
a norm. However, $\alpha\geq 0$ is less natural but ensures that the objective function is bounded from below. The presence of the constraint
$\alpha\geq 0$ does not affect the feasibility. Note that to reduce the size of the problem, we can take $q^i=0$ and get an homogeneous PQL function. 

Now, we can explain the motivation of $(1,0_{1\times d})$ in Eq.~\eqref{homogeneouspol}. It would be more natural to express $\bHom{P,c}$ as $(c \ -P)$. 
However, when we replace the cone-copositivity constraints by right-hand-side of Eq.~\eqref{eqcorollary} and by doing this we allow symmetry as it is shown
in Example~\ref{reason_of_one} and the vector $(1,0_{1\times d})$ aims to break it.
\begin{example}[Why is there $(1,0_{1\times d})$ in $\bHom{P,c}$?]
\label{reason_of_one}
Consider $X=\{x\in\rr\mid x\leq 1\}$. Let $u(x)=(1,x)$,
and $M=(1\ -1)$ ($\bHom{1,1}$ without $(1,0)$). Then $X=\{x\mid M u(x)^\intercal \geq 0\}$. 

Now let $W\geq 0$ and define $X'=\{x\mid u(x) M^\intercal W M u(x)^\intercal \geq 0\}$.
Since $u(x) M^\intercal W M u(x)^\intercal=W u(x)M^\intercal M u(x)^\intercal=2W(1-x)^2$, $X'=\rr$ for all $W\geq 0$.

Now let us take $E=\bHom{1,1}$ and let $W=\begin{psmallmatrix} w_1 & w_3 \\ w_3 & w_2\end{psmallmatrix}$ with $w_1,w_2,w_3\geq 0$ 
and define $\overline{X}=\{x\mid u(x) E^\intercal W E u(x)^\intercal \geq 0\}$.
Hence, $u(x) E^\intercal \begin{psmallmatrix} w_1 & w_3 \\ w_3 & w_2\end{psmallmatrix} E u(x)^\intercal=w_1+2w_3(1-x)+w_2(1-x)^2$.
Taking for example $w_2=w_1=0$ and $w_3>0$ implies that $\overline{X}=X$.
\end{example}
\begin{proposition}
\label{PQLprop}
Assume that Problem~\eqref{SDPlyap} has a feasible solution $(\Pq,\mathcal{W},\mathcal{U},\mathcal{Z},\alpha,\beta)$. Then:
\begin{enumerate}
\item The family $\Pq$ defines a PQL;
\item There exists $(\Pq,\mathcal{W},\mathcal{U},\mathcal{Z},\alpha,\beta)$ satisfiying~\eqref{boundedeqrelax},~\eqref{stabilityeqrelax}
and~\eqref{initialeqrelax} if and only if Problem~\eqref{SDPlyap} is feasible;
\item For all $(i,j)\in \overline{\Sw}$,
\[
\begin{array}{ll}
& {F^i}^\intercal\Mat(\Id,0,0)F^i\\
\preceq & \Mat(P^i,2q^i,-\alpha)+\Mat(0,0,\beta)\\
&-{E^{ij}}^\intercal\left( \begin{pmatrix}0_{n_i}  & 0_{n_i,n_j} \\
0_{n_j,n_i} & \pos{W^j}+\psd{W^j}\end{pmatrix}+\pos{U^{ij}}+\psd{U^{ij}}\right) E^{ij}\enspace ; 
\end{array}
\]
\item We have $\displaystyle{\sup_{x\in X^0} \norm{x}_2^2} \leq \beta$;
\item If $(\Pq,\mathcal{W},\mathcal{U},\mathcal{Z},\alpha,\beta)$ is optimal and $\alpha>0$ then $\displaystyle{\sup_{x\in X^0} L(x)}=\alpha$.
\end{enumerate}
\end{proposition}
\begin{proof}
In appendix.
\end{proof}

\section{Sublevel Modelisation}
\label{sublevel}
In Def.~\ref{computablepiecewise}, $\beta$ is an upper bound on the Euclidian norm
of the state variable. We do not have a precise upper bound on each coordinate considered separetely neither a precise upper bound 
on the state variable considering a specific cell.   
To obtain tigher bounds on the state variables, we intersect $S_{\alpha}$ with other sublevel sets. 
In~\cite{DBLP:conf/hybrid/RouxJGF12}, the authors propose to combine classical quadratic Lyapunov function sublevels and the square of variables. 
In this paper, we apply this technique replacing classical Lyapunov functions by PQL functions.
Thus we are interested in a set $V$ of the form $V=S_{\alpha}\cap \cup_{i\in \ind} \{y\in X^i \mid y_l^2\leq \beta_l^i, l=1,\ldots,d\}$.
The computation of $V$ is thus reduced to compute $\beta_l^i$. In verification of programs, the method is called a \emph{templates domain abstraction}
(for more background~\cite{dblp1968095}).

From Eq.~\eqref{reach}, $\rea=\map{}(\rea)\cup X^0$. We introduce the map $F:\wp(\rd)\mapsto \wp(\rd)$ defined by:
 \[
 C\mapsto F(C):=\map{}(C)\cup X^0\enspace .
 \]
Hence $\rea$ is the smallest fixed point of $F$ in the sense of if $C=F(C)$ then $\rea\subseteq C$. 
 From Tarski's theorem~\cite{tarski1955}, since $F$ is monotone on $\wp(\rd)$, then:
 \begin{equation}
 \label{reatarski}
 \rea=\inf\{C\in\wp(\rd)\mid F(C)\subseteq C\};
 \end{equation}
 Consequently, if we take any subset $C$ such that $F(C)\subseteq C$ then $\rea\subseteq C$.
 We propose to consider a restricted family of subsets $C$ parameterized by
 $\betab\in\rr^{d+1}$:
 \[
 C(\betab):=\{x\in \rd\mid\forall\, k\in [d],\ x_k^2\leq \betab_k, L(x)\leq \betab_{d+1}\}  
 \]
 where $L$ is a PQL function of $\pws$. We define:
\[
\forall\, k\in [d],\ X_k^0=\sup_{y\in X^0} y_k^2\text{ and }X_{d+1}^0=\sup_{y\in X^0} L(y)
\]
We also define for all $(i,j)\in\overline{\Sw}$ and for all $\betab\in\rr^{d+1}$:
\[
\begin{array}{lc}
\displaystyle{\forall\, k\in[d],}&\ \displaystyle{\sha{F_{ij,k}}(\betab)=\sup_{\substack{\forall\, k\in [d],\ x_k^2\leq \betab_k,\\ L^i(x)\leq \betab_{d+1},\ x\in X^{ij}}} (A_{k \cdot}^i x+b_k^i)^2}\\
\text{ and }&\\
& \displaystyle{\sha{F_{ij,d+1}}(\betab)=\sup_{\substack{\forall\, k\in [d],\ x_k^2\leq \betab_k,\\ L^i(x)\leq \betab_{d+1},\ x\in X^{ij}}} L^j(A^i x+b^i)}
\end{array}
\]
and finally, we define for all $\betab\in\rr^{d+1}$:
\[
\forall\, l\in[d+1],\ \displaystyle{\sha{F_{l}}(\betab)=\sup\{\sup_{(i,j)\in \overline{\Sw}}\sha{F_{ij,l}}(\betab),X_l^0\}} 
\]
and $\sha{F}(\betab)=(\sha{F_1}(\betab),\ldots,\sha{F_{d+1}}(\betab))$.
\begin{proposition}
\label{modelprop}
The following statements hold: 
\begin{enumerate}
\item $F(C(\betab))\subseteq C(\betab)\iff\sha{F}(\betab)\leq \betab$;
\item $\rea\subseteq\inf\{C(\betab)\mid \betab\in\rr^{d+1} \text{ s.t. }\sha{F}(\betab)\leq \betab\}$;
\item For all $l\in [d+1]$, $\sha{F_{ij,l}}(\betab)$ is the optimal value of quadratic program;
\item For all $k\in[d]$, $X_k^0=\displaystyle{\max\{(\inf_{x\in X^0} x_k)^2,(\sup_{x\in X^0} x_k)^2\}}$ and if $L$ is constructed from an optimal solution $(\Pq,\mathcal{W},\mathcal{U},\mathcal{Z},\alpha,\beta)$ of~\eqref{SDPlyap} such that $\alpha>0$, then $X_{d+1}^0=\alpha$.
\end{enumerate}
\end{proposition}
\begin{proof}
In appendix.
\end{proof}

\section{Policy Iteration Algorithm}
\label{policy}
Now, we assume that Problem~\eqref{SDPlyap} has an optimal 
solution $(\Pq,\mathcal{W},\mathcal{U},\mathcal{Z},\alpha,\beta)$ with $\alpha>0$ and let $L$ be the associated PQL function. 

From Prop.~\ref{modelprop}, to evaluate $\sha{F_{ij,l}}(\betab)$ is equivalent to solve a quadratic maximisation problem which is known to be NP-Hard~\cite{Vavasis199073}. So we propose to compute instead a safe overapproximation using Lagrange duality and semi-definite programming. 
\subsection{Relaxed functional} 
In this subsection, we define the function on which we compute fixed point.
Let $(i,j)\in\overline{\Sw}$, $\betab\in\rr^{d+1}$.

For all $k\in[d]$, we write $\Mat_k$ for $\Mat(x\mapsto x_k^2)$ and for all $i\in\ind$, $\Mat_L^i$ for $\Mat(L^i)$. The matrix $\mathbf{N}\in\mt{(d+1)}{(d+1)}$ is defined as follows: $\mathbf{N}_{1,1}=1$ and $\mathbf{N}_{l,m}=0$ for all $(l,m)\in [d+1]^2\backslash\{(1,1)\}$.  

Let, $\lambda\in\rr_+^{d+1}$, $Y\in\Sp{n_{ij}}$ and $Z\in \Sdp{n_{ij}}$. Then:
\begin{equation}
\begin{array}{l}
\Phi_{ij,k}(\lambda,Y,Z)=\\
\displaystyle{{F^i}^\intercal \Mat_k F^i-\sum_{l=1}^d \lambda_l \Mat_l -\lambda_{l+1}\Mat_L^i+{E^{ij}}^\intercal(Y+Z)E^{ij}}\\
\Phi_{ij,d+1}(\lambda,Y,Z)=\\
\displaystyle{{F^i}^\intercal \Mat_L^j F^i-\sum_{l=1}^d \lambda_l \Mat_l -\lambda_{l+1}\Mat_L^i+{E^{ij}}^\intercal(Y+Z)E^{ij}}
\end{array}
\end{equation}
For all $l\in[d+1]$, for all $\betab\in\rr_+^{d+1}$:
 \begin{equation}
 \label{relaxedSDP}
\begin{array}{ll}
\rel{F}_{ij,l}(\betab)=&\\
\displaystyle{\inf_{\lambda,\eta,Y,Z}} & \eta\\
\st & \left\{
\begin{array}{l}
(\eta-\sum_{k=1}^{d+1} \lambda_k \betab_k)\mathbf{N}-\Phi_{ij,l}(\lambda,Y,Z)\succeq 0,\\
\lambda\in\rr_+^{d+1},\ \eta\in\rr,\ Y\geq 0,\ Z\succeq 0
\end{array}
\right.
\end{array}
\end{equation}
\[
\rel{F}_{l}(\betab)=\sup\{ \sup_{(i,j)\in\overline{\Sw}} \rel{F_{ij,l}}(\betab),X_l^0\}
\]
and $\rel{F}(\betab)=(\rel{F}_1(\betab),\ldots,\rel{F}_{d+1}(\betab))$.

\begin{proposition}[Safe overapproximation]
\label{abstractprop}
The following assertions are true:
\begin{enumerate}
\item For all $l\in[d+1]$, $\rel{F_l}$ is the optimal value of a SDP program;
\item $\sha{F}\leq \rel{F}$\enspace . 
 \end{enumerate}
\end{proposition}
\begin{proof}
In appendix.
\end{proof}
\begin{lemma}
\label{relaxedviainfsup}
Let $(i,j)\in\overline{\Sw}$, $l\in[d+1]$ and $\betab\in\rr^{d+1}$. Then:
\[
\rel{F_{ij,l}}(\betab)=\inf_{\lambda\in\rr_+^{d+1}} F_{ij,l}^{\lambda}(\betab)
\]
where 
\begin{equation}
\label{affine}
F_{ij,l}^{\lambda}(\betab)=\sum_{m=1}^{d+1} \lambda_m \betab_m+\inf_{\substack{Y\geq 0\\ Z\succeq 0}} \sup_{x\in\rd} \ux^\intercal \Phi_{ij,l}(\lambda,Y,Z)\ux 
\end{equation}
\end{lemma}
\begin{proposition}
\label{relaxedprop}
Let $(i,j)\in\overline{\Sw}$, $l\in [d+1]$, $\lambda\in\rr_+^{d+1}$.
The following statements are true:
\begin{enumerate}
\item $F_{ij,l}^{\lambda}$ is affine;
\item $F_{ij,l}^{\lambda}$, $\rel{F_{ij,l}}$ and $\rel{F_l}$ are monotone;
\item $\rel{F_{ij,l}}$ and $\rel{F_l}$ are upper semi-continuous. 
\end{enumerate}
\end{proposition}
\begin{proof}
In appendix.
\end{proof}
To be able to perform a new step in policy iteration, we need a selection property. In our case, the selection property relies on the existence of an optimal dual solution.
\begin{definition}[Selection property]
Let $(i,j)\in\overline{\Sw}$ and $l\in [d+1]$. We say that $\betab\in\rr^{d+1}$ satisfies the selection property if 
there exists $\lambda\in\rr_+^{d+1}$ such that:
\begin{equation}
\label{selectth}
\rel{F_{ij,l}}(\betab)=F_{ij,l}^{\lambda}(\betab)
\end{equation}
We define:
\[
\Sol{(i,j),l,\betab}:=\{\lambda\in \rr_+^{d+1}\mid \rel{F_{ij,l}}(\betab)=F_{ij,l}^{\lambda}(\betab)\}
\]
and 
\[
\begin{array}{l}
\FS:=\\
\{\betab\in\rr^{d+1}\mid \forall\, (i,j)\in\o\overline{\Sw},\forall\, l\in [d+1],\Sol{(i,j),l,\betab}\neq \emptyset\}\enspace.
\end{array}
\]
\end{definition}
\begin{corollary}
\label{lemma2}
Let $(i,j)\in\overline{\Sw}$, $l\in [d+1]$ and $\betab\in\FS$. Now let $\overline{\lambda}\in \Sol{(i,j),\betab,p}$, then: 
\[
\displaystyle{\inf_{\substack{Y\geq 0\\ Z\succeq 0}} \sup_{x\in\rd} \ux^\intercal \Phi_{ij,l}(\lambda,Y,Z)\ux}=
\rel{F_{ij,l}}(\betab)-\sum_{m=1}^{d+1} \overline{\lambda}_m \betab_m\enspace .
\] 
\end{corollary}
Let $(i,j)\in \overline{\Sw}$, $l\in [d+1]$ and $\betab\in \FS$. From Corollary~\ref{lemma2}, for all $\lambda\in\Sol{(i,j),l,\betab}$, we can rewrite for all $v\in\rr^{d+1}$ as follows:
\begin{equation}
\label{affine2}
F_{ij,l}^{\lambda}(v)=\sum_{m=1}^{d+1} \lambda_m v_m+\rel{F_{ij,l}}(\betab)-\sum_{m=1}^{d+1} \overline{\lambda}_m \betab_m
\end{equation}
We remark that $F_{ij,l}^{\lambda}(\betab)=\rel{F_{ij,l}}(\betab)$.

From the first statement of Prop.~\ref{modelprop} and the second assertion of Prop.~\ref{abstractprop}, the most precise 
overapproximation of $\rea$ (with these quadratic functions) is given by:
\[
\overline{\betab}=\inf\{\betab\in\rr^{d+1}\mid \rel{F}(\betab)\leq \betab\}
\]
From Tarski's theorem, $\overline{\betab}$ is the (finite) smallest fixed point of $\rel{F}$. So we are looking for 
the smallest fixed point of $\rel{F}$. The smallest seems difficult to obtain and since any vector $\betab$ such that $\rel{F}(\betab)\leq \betab$ furnishes 
a valid but less precise overapproximation of $\rea$, we perform a policy iteration until a fixed point is reached.       
\subsection{Policy definition}
A policy iteration algorithm can be used to solve a fixed point equation for a 
monotone function written as an infimum of a family of simpler 
monotone functions, obtained by selecting {\em policies},
see~\cite{costan2005policy,gaubert2007static} for more background. The idea 
is to solve a sequence of fixed point problems involving the simple functions.
In the present setting, we look for a representation of the relaxed
function:
\begin{align}
\forall\, (i,j)\in\overline{\Sw},\ \forall\, l\in [d+1],\ \rel{F_{ij,l}}= \inf_{\pi\in \Pi}F_{ij,l}^\pi \label{e-def-select}
\end{align}
where the infimum is taken over a set $\Pi$ whose elements $\pi$ are called {\em policies}, and
where each function $F^\pi$ is required to be monotone. The correctness
of the algorithm relies on a selection property, meaning in the present
setting that for each argument $((i,j),l,\betab)$ there must exist a
policy $\pi$ such that $\rel{F_{ij,l}}(\betab)= F^\pi_{ij,l}(\betab)$. The idea
of the algorithm is to start from a policy $\pi^0$, compute the smallest
fixed point $\betab$ of $F^{\pi^0}$, evaluate $\rel{F}$ at point $\betab$,
and, if $\betab\neq \rel{F}(\betab)$, determine the new policy using the selection property at point $\betab$.

Let us now identify the policies. 
Lemma~\ref{relaxedviainfsup} shows that for all $l\in [d+1]$,
$\rel{F_{ij}}$ can be written as the infimum of the family of affine functions $F_{ij}^{\lambda}$, 
the infimum being taken over the set of $\lambda\in\rr_+^{d+1}$.
When $\betab\in\FS$ is given, choosing a policy $\pi$ consists
in selecting, for each $(i,j)\in\overline{\Sw}$ and for all $l\in[d+1]$, a vector $\lambda\in \Sol{(i,j),l,\betab}$. 
We denote by $\pi_{ij,l}(\betab)$ the value of $\lambda$ chosen by the policy $\pi$. 
Then, the map $F_{ij,l}^{\pi_{ij,l}}$ in Equation~\eqref{e-def-select} is obtained by replacing $\rel{F_{ij,l}}$ by 
$F_{ij,l}^\lambda$ appearing in Eq.~\eqref{affine2}. 

Finally, we define, for all $l\in [d+1]$:
\[
F_l^\pi(\betab)=\sup\{\sup_{(i,j)\in \overline{\Sw}} F_{ij,l}^{\pi_{ij,l}}(\betab),X_l^0\}
\]
and $F^\pi=(F_1^{\pi},\ldots,F_{d+1}^{\pi})$.

Now, we can define concretely the policy iteration algorithm at Algorithm~\ref{PIQua}.
\begin{algorithm}
\caption{Policy Iteration with PQL functions}\label{PIQua}
\begin{itemize}
 \item[1] Choose $\pi^0\in\Pi$,  $k=0$.
 \item[2] Define $F^{\pi^k}$ by choosing $\lambda$ according to policy $\pi^k$ using Eq.~\eqref{affine2}.
 \item[3] Compute the smallest fixed point $\betab^k$ in $\rr^{d+1}$ of $F^{\pi^k}$.
 \item[4] If $\betab^k\in\FS$ continue otherwise return $\betab^k$.
 \item[5] Evaluate $\rel{F}(\betab^k)$, if $\rel{F}(\betab^k)=\betab^k$ 
return $\betab^k$ otherwise take $\pi^{k+1}$ s.t. $\rel{F}(\betab^k)= F^{\pi^{k+1}}(\betab^k)$.
Increment $k$ and go to 2.  
\end{itemize}
\end{algorithm}
\subsection{Some details about Policy Iteration algorithm}
\paragraph*{Initialization}
Policy iteration algorithm needs an initial policy. Recall that 
we have assumed that $L$ was computed from an optimal solution $(\Pq,\mathcal{W},\mathcal{U},\mathcal{Z},\alpha,\beta)$ of Problem~\eqref{SDPlyap} such that $\alpha>0$.
The first policy is given by a choice of an element in $\Sol{(i,j),l,w^0}$ where $w^0$ is defined by:
\begin{equation}
\label{initialpoint}
\forall\, k\in [d],\ \betab_k^0=\beta,\ w_{d+1}^0=\alpha
\end{equation}
with $\alpha$ and $\beta$ are extracted from $(\Pq,\mathcal{W},\mathcal{U},\mathcal{Z},\alpha,\beta)$.
\begin{proposition}
\label{initialprop}
The vector $\betab^0$ satisfies $\rel{F}(\betab^0)\leq \betab^0$.
\end{proposition}
\begin{proof}
In appendix.
\end{proof}
\paragraph*{Smallest fixed point computation associated to a policy}
For the third step of Algorithm~\ref{PIQua}, using Lemma ~\ref{relaxedviainfsup}, 
$F^{\pi}$ is monotone and affine, we compute the smallest fixed point of $F^{\pi}$ 
by solving the following Linear Program see ~\cite[Section 4]{gaubert2007static}:
\begin{equation}
\label{LPfix}
\min\left\{\sum_{k=1}^{d+1} w_k\ \mathrm{s.t.}\ F^{\pi}(w)\leq w\right\}
\end{equation}
\paragraph*{Convergence}
In~\cite{adjepolicy}, it is proved that policy iterations in the quadratic setting converges towards 
a fixed point of our relaxed functional. Here we establish a similar result (Th.~\ref{decro}).
Combined with Prop.~\ref{abstractprop}, this fixed point provides a 
safe overapproximation of the reachable values set.

Let consider the sequence $(w^l)_{l\geq 0}$ computed by Algorithm~\ref{PIQua}.
If for some $l\in\mathbb{N}$, $w^l\notin\FS$ and $w^{l-1}\in\FS$, then we set $w^k=w^l$ for all $k\geq l$.
\begin{theorem}
\label{decro}
The following assertions hold:
\begin{enumerate}
\item For all $l\in\mathbb{N}$, $\rel{F}(w^l)\leq w^l$;
\item The sequence $(w^l)_{l\geq 0}$ is decreasing. Moreover for all $l\in\nn$ such that $w^{l-1}\in\FS$
either $w^l=w^{l-1}$ and $\rel{F}(w^l)=w^l$ or $w^l<w^{l-1}$;
\item For all $l\in\mathbb{N}$, for all $k\in [d+1]$, $X_k^0\leq w_k^l\leq w_k^0$;
\item The limit $w^\infty$ of $(w^l)_{l\geq 0}$ satisfies: $\rel{F}(w^\infty)\leq w^\infty$. Moreover if $\forall\, k\in\nn$, $w^k\in\FS$ then $\rel{F}(w^\infty)=w^\infty$.
\end{enumerate}
\end{theorem}
\begin{proof}
(1)~From Prop.~\ref{initialprop}, $\rel{F}(w^0)\leq w^0$. Now, let $l>0$ and assume $w^{l-1}\in\FS$, there exists $\pi^l$ such that, $F^{\pi^l}(w^l)=w^l$ and since 
$\rel{F}=\inf_{\pi} F^{\pi}$, we get $\rel{F}(w^l)\leq F^{\pi^l}(w^l)=w^l$. If $w^{l-1}\notin\FS$, then there exists $k\in\nn$, 
$k\leq l-1$ such that $w^{k-1}\in\FS$ and $w^l=w^k$, and thus by the latter argument we have $\rel{F}(w^k)\leq w^k$.

(2)~Let $l\in\nn$, if $w^{l-1}\notin\FS$, $w^l=w^{l-1}$. Now suppose $w^{l-1}\in\FS$. There exists $\pi^l\in\Pi$ such that
$\rel{F}(w^{l-1})=F^{\pi^l}(w^{l-1})\leq w^{l-1}$ and since $w^l$ is the smallest element of $\{v\in\rr^{d+1}\mid F^{\pi^{l}}(v)\leq v\}$
then $w^l\leq w^{l-1}$. Now if $w^l=w^{l-1}$, $\rel{F}(w^{l-1})=\rel{F}(w^{l})=F^{\pi^l}(w^{l-1})= F^{\pi^l}(w^{l})=w^l=w^{l-1}$.

(3)~From the second assertion, for all $l\in\nn$, $w^l\leq w^0$. Moreover, for all $k\in [d+1]$,  
$X_k^0\leq \sha{F}_k(w^l)\leq \rel{F}_k(w^l)\leq w_k^l$.

(4)~First, $w^\infty$ exists since $(w_l)_{l\in\nn}$ is decreasing and bounded from below (third assertion). Then, for all $l\in\nn$, $w^\infty\leq w^l$ and thus since $\rel{F}$ is monotone 
(Prop.~\ref{relaxedprop}) $\rel{F}(w^\infty)\leq \rel{F}(w^l)\leq w^l$. Taking the infimum over $l$, we get $\rel{F}(w^\infty)\leq w^\infty$.
 Now we prove that $w^{\infty}\leq \rel{F}(w^{\infty})$. Let $l\in\nn$. By assumption, $w^l\in\FS$ and then, there exists $\pi^{l+1}\in\Pi$ such that 
$F^{\pi^{l+1}}(w^l)=\rel{F}(w^l)$. Moreover, $w^{l+1}\leq w^l$ and since $F^{\pi^{l+1}}$ is monotone (Prop.~\ref{relaxedprop}): 
$w^{l+1}=F^{\pi^{l+1}}(w^{l+1})\leq F^{\pi^{l+1}}(w^l)=\rel{F}(w^l)$. 
Now by taking the infimum on $l$, we get 
$w^{\infty}=\inf_l w^{l+1}=\inf_l w^l\leq \inf_l \rel{F}(w^l)$. Finally 
since $\rel{F}$ is upper semicontinuous (third point of Prop.~\ref{relaxedprop}), then $\inf_k \rel{F}(w^k)=\limsup_k \rel{F}(w^k)\leq  \rel{F}(\lim_k w^k)= \rel{F}(w^{\infty})$. We conclude that $w^{\infty}\leq \rel{F}(w^{\infty})$.
\end{proof}

\section{Example}
\label{example}
\subsection{Example from~\cite{912814} slighty modified}
Consider the followinf PWA: $X^0=[-1,1]\times [-1,1]$, and, for all $k\in\nn$:
\[
\begin{array}{l}
x_{k+1}= 
\left\{ 
 \begin{array}{lr}
   A^1 x_k& \text{ if } x_{k,1} \geq 0\text{ and } x_{k,2} \geq  0\\
   A^2 x_k & \text{ if } x_{k,1} \geq 0\text{ and } x_{k,2} < 0\\
   A^3 x_k& \text{ if } x_{k,1} < 0\text{ and } x_{k,2} < 0\\
   A^4 x_k &\text{ if } x_{k,1} < 0\text{ and } x_{k,2} \geq 0
 \end{array}
 \right.
 \end{array}
\]
with 
\[
\begin{array}{c}
A^1=\begin{pmatrix}
    -0.04 & -0.461\\
    -0.139 & 0.341
    \end{pmatrix},\ 
A^2=\begin{pmatrix}
    0.936 & 0.323\\
    0.788 & -0.049
   \end{pmatrix}\\ 
A^3=\begin{pmatrix}
    -0.857 & 0.815\\
    0.491 & 0.62
    \end{pmatrix},\ 
A^4=\begin{pmatrix}
    -0.022 & 0.644\\
    0.758 & 0.271
    \end{pmatrix} 
\end{array}
\]
Then, we have $X^1=\rr_+\times\rr_+$, $X^2=\rr_+\times \rr_{-}^*$, $X^3=\rr_-^*\times \rr_-^*$ and $X^4=\rr_-^*\times \rr_+$.

From Prop.~\ref{propSwitches}, $\In=\{1,2,3,4\}$ and $\overline{\Sw}=\{(i,j)\mid S(i,j)=1\}$
with $S=\begin{psmallmatrix} 1 & 0 & 1 &1\\1 & 0 & 0 & 1\\ 0& 1 & 1 & 0\\1 & 1& 0 & 0\end{psmallmatrix}$.

By solving Problem~\ref{SDPlyap}, we get a (optimal) PQL function $L$ characterized by the following matrices:
\[
\begin{array}{c}
P^1=
\begin{pmatrix}
  1.1178 & -0.1178 \\
 -0.1178 &  1.1178 \\
\end{pmatrix},\
P^2=
\begin{pmatrix}
 1.5907   & 0.5907\\
    0.5907  &  1.5907
\end{pmatrix},\\
P^3=
\begin{pmatrix}
    1.3309  & -0.3309\\
   -0.3309  &  1.3309
\end{pmatrix},\
P^4=
\begin{pmatrix}
 1.2558   & 0.2558\\
    0.2558 &   1.2558
\end{pmatrix}
\end{array}
\]
Since $\alpha=\beta=2$, then $\rea\subseteq \{x\in\rr^2\mid L(x)\leq 2\}\subseteq \{x\in\rr^2\mid \norm{x}_2^2\leq 2\}$. The sets 
$\rea$ (discretized version) and $\{x\in\rr^2\mid L(x)\leq 2\}$ are depicted at Figure~\ref{firstinv}.
Then we enter into policy iteration algorithm. From Equation~\eqref{initialpoint}, we define $w^0$ by:
\[
w_1^0=2.0000,\ 
w_2^0=2.0000,\ 
w_3^0=2.0000
\]  
Then we compute the image of $w^0$ by the relaxed semantics $\rel{F}(w^0)$ using semidefinite programming (see Eq.~\eqref{relaxedSDP}). We check that $w^0$ is not a fixed point of $\rel{F}$ and then the initial policy $\pi^0((i,j),l,w^0)$ is the vector $\lambda$ extracted from the optimal solutions $(\lambda,Y,Z)$ of the semidefinite programs involved in the computation of $\rel{F}(w^0)$. For example, for $(1,3)\in\overline{\Sw}$ and $l=1$, $\pi^0((1,3),1,w^0)=(0.0000, 0.0000, 0.0430)^\intercal$, where the first two zeros are the Lagrange multipliers associated to $\Mat_1$ and $\Mat_2$
and $0.0430$ is the Lagrange multiplier associated to $\Mat(L^1)$. 
\comment{
Note that the solver also provides a nonnegative matrix:
\[ 
Y=\begin{pmatrix}
  0.0000 &  0.0000 &  0.0000 &  0.0000 &  0.0000 &  0.0001 \\
  0.0000 &  0.0000 &  0.0000 &  0.0000 &  0.0000 &  0.0193 \\
  0.0000 &  0.0000 &  0.0000 &  0.0000 &  0.0000 &  0.4096 \\
  0.0000 &  0.0000 &  0.0000 &  0.0000 &  0.0000 &  0.0001 \\
  0.0000 &  0.0000 &  0.0000 &  0.0000 &  0.0000 &  0.3475 \\
  0.0001 &  0.0193 &  0.4096 &  0.0001 &  0.3475 &  0.3452 \\
\end{pmatrix}
\]
and a semidefinite matrix:
\[ 
Z=\begin{pmatrix}
  0.3640 &  0.0000 &  0.0001 & -0.3640 & -0.0002 &  0.0003 \\
  0.0000 &  0.0226 & -0.0315 & -0.0000 & -0.0428 & -0.0877 \\
  0.0001 & -0.0315 &  0.2109 & -0.0001 & -0.2389 &  0.1726 \\
 -0.3640 & -0.0000 & -0.0001 &  0.3640 &  0.0002 & -0.0003 \\
 -0.0002 & -0.0428 & -0.2389 &  0.0002 &  0.6150 &  0.0763 \\
  0.0003 & -0.0877 &  0.1726 & -0.0003 &  0.0763 &  0.6571 \\
\end{pmatrix}
\]
}
We compute the smallest fixed point associated to $\pi^0$ using the linear program~\eqref{LPfix}:
\[
w_1^1=1.1036,\ 
w_2^1=1.2443,\ 
w_3^1=2.0000
\]
Moreover, at each step $k$, policy iterations provides auxiliary values which represent the overapproximations of the polyhedra $\rea\cap X^i\cap {A^i}^{-1}(X^j)$
by ellipsoids of the form $\{x\in\rr^2\mid x_{1}^2\leq w_{ij,1}^k,\ x_{2}^2\leq w_{ij,2}^k,\ L(x_{1},x_{2})\leq w_{ij,3}^k\}$. For example, for $k=0$:
\[
\begin{array}{ccc}
w_{11,1}=0.0000,\ 
w_{11,2}=0.0000,\ 
w_{11,3}=0.0000
\\
w_{13,1}=0.0573,\ 
w_{13,2}=0.0213,\ 
w_{13,3}=0.0213
\\
w_{14,1}=0.3012,\ 
w_{14,2}=0.1447,\ 
w_{14,3}=0.1447
\end{array}
\]
Note that we found that for $(i,j)=(1,1)$, $w_{ij,1}^1=w_{ij,2}^1=w_{ij,3}^1=0$ which means that $\rea\cap X^1\cap {A^{1}}^{-1}(X^1)$ is reduced to the singleton $(0,0)$. The invariant found is depicted at Figure~\ref{finalinv}.
\begin{figure}
\centering
\captionsetup[subfigure]{width=3.5cm}
\subfloat[First overapproximation found by~\eqref{SDPlyap}]{\label{firstinv}
\centering
 \includegraphics[width=.3\textwidth]{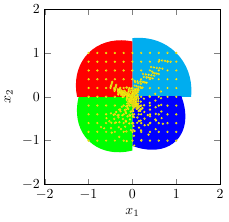}
}
\hspace{30pt}
\subfloat[Final overapproximation found by policy iterations]
{\label{finalinv}
\centering
\includegraphics[width=.3\textwidth]{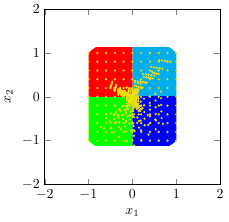}
}
\caption{(Discretized) $\rea$ in yellow and initial and last overapproximations of $\rea$.}
\label{invs}
\end{figure}
Finally, we find after two iterations that for all $k\in\nn$, $x_{1,k}^2\leq 1,\ x_{2,k}^2\leq 1.2443$ and $L(x_{1,k},x_{2,k})\leq 2$. 
\subsection{A (piecewise) affine example} 
 We now consider the following PWA: $X^0=[0,3]\times [0,2]$ and for all $k\in\nn$:
 \[
\begin{array}{l}
x_{k+1}= 
\left\{ 
 \begin{array}{lr}
   A^1 x_k+b^1& \text{ if } T(x_{k}) < c\\
   A^2 x_k+b^2 & \text{ if } T(x_{k}) \geq c\\
 \end{array}
 \right.
 \end{array}
\]
with
\[ 
\begin{array}{c}
A^1=\begin{pmatrix}
  0.4197 & -0.2859 \\
  0.5029 &  0.1679 \\
\end{pmatrix},\quad b^1=\begin{pmatrix}
  2.0000 \\  5.0000 
\end{pmatrix},\\
A^2=\begin{pmatrix}
 -0.0575 & -0.4275 \\
 -0.3334 & -0.2682 \\
\end{pmatrix},\quad b^2=\begin{pmatrix}
 -4.0000 \\  4.0000
\end{pmatrix}\\
\\
T=\begin{pmatrix}  3.0000 &  8.0000\end{pmatrix}\text{ and } c=-3.0000
\end{array}
\]
By Prop.~\ref{propSwitches}, $\overline{\Sw}=\ind^2=\{(1,1),(1,2),(2,1),(2,2)\}$ and $\In=\{2\}$. Using Problem~\eqref{SDPlyap}, 
we compute the PQL function $L$ characterized by:
\[
\begin{array}{c}
P^1=
\begin{pmatrix}
2.9888 & -1.7890 \\
-1.7890 &  8.0295
\end{pmatrix},\quad q^1=\begin{pmatrix}
-14.7283\\
-94.1347
\end{pmatrix}\\
\text{and}\\
P^2=\begin{pmatrix}
2.7192 &  2.0930 \\
2.0930 &  6.1110
\end{pmatrix},\quad q^2=\begin{pmatrix}
5.5737\\
-16.4198
\end{pmatrix}\\
\end{array}
\]
and the invariant found is $\{x\in\rr^2\mid L(x)\leq 58.1165\}$ and an upper bound over the square Euclidian norm of the state variable is $286.4932$.
We run the policy iteration to get finally after 4 iterations the following bound vector:
\[
w_1=41.8956,\ 
w_2=31.4449,\  
w_3=58.1165
\]
corresponding to the invariant set $w\{x\in\rr^2\mid x_i^2\leq w_i,\ L(x)\leq w_3\}$. 

We obtain interesting information during policy iterations running.
At step $k=0$, when we select the initial policy, the SDP solver returns for all $l=1,2,3$, $\rel{F}_{11,l}(w^0)=-\infty$ and from 
Prop.~\ref{abstractprop} this implies that $\sup_{x\in \rea\cap X^1\cap {f^1}^{-1}(X^1)} p(A^1x+b^1)$
is not feasible hence $(1,1)\notin\Sw$.
At iteration step $k=1$, the SDP solver provides for all $l=1,2,3$, $\rel{F}_{21,l}(w^1)=-\infty$ and from 
Prop.~\ref{abstractprop} this implies that $\sup_{x\in \rea\cap X^1\cap {f^2}^{-1}(X^2)} p(A^2x+b^2)$
is not feasible hence $(2,1)\notin\Sw$. 
Finally, $\Sw\subseteq \{(1,2),(2,2)\}$. Recalling that $1\notin \In$, 
we conclude that the system state variable only stays in $X^2$ and thus the system is actually
equivalent to a \emph{constrained affine system}. This information is computed \emph{automatically}.  
\comment{
We run the policy iteration to get finally after 3 iterations the following bound vector
\[
w(x_{1}^2)=1.0000,\ 
w(x_{2}^2)=1.8255,\ 
w(x_{3}^2)=1.0456,\ 
w(x_{4}^2)=1.1831,\ 
w(L)=4.0000
\]
corresponding to the invariant set $w^\star=\{x\in\rr^4\mid x_i^2\leq w(x_i^2),\ L(x)\leq w(L)\}$. 

The interesting information we get during the policy iteration is that actually 
the set of possible switches is included in $\{(1,1),(1,2),(2,1)\}$ i.e. we have found that we cannot reach 
the second cell from the second cell according to the second law $A^2$. 

This information is actually available when we evaluate the policy $\pi^1$ for the coordinate $(2,2)$. 
Indeed, the initial policy yields to the bound function $w^0$:
\[
w^0(x_{1}^2)= w^0(x_{2}^2)= w^0(x_{3}^2)= w^0(x_{4}^2)= w^0(L)=4.0000
\]
and ${w^0}^\star=\{x\in\rr^4\mid x_i^2\leq w^0(x_i^2),\ L(x)\leq w^0(L)\}$.
When we evaluate $\rel{F}_{22}(w^0)(p)$ whatever the template $p\in\pp$, $\rel{F}_{22}(w^0)(p)=-\infty$ and from 
Proposition~\ref{abstractprop} this implies that 
\[
 \mybrackets{\sha{F_{22}}}(w^0)(p)=\sup_{x\in {w^0}^{\star}\cap X^2\cap {f^2}^{-1}(X^2)} p(A^2x)
\]
is not feasible hence $(2,2)\notin\Sw$.
}

\section{Conclusion and Future Works}
\label{conclusion}
We have developed a method to compute \emph{automatically} by semi-definite programming precise bounds over the reachable values set of a piecewise affine system. 
The method combines piecewise quadratic Lyapunov functions to generate a first overapproximation and policy iterations used to reduce the initial overapproximation.

Future works could be to design a repartitioning method in order to improve the feasibility of Problem~\eqref{SDPlyap}. Morevoer, we can think 
of apply the method to maximize a quadratic form over the reachable values set.  

Also, we conjecture that the presented policy iterations algorithm provides the most precise overapproximation considering bounding the square of coordinates variables.
To reduce these bounds we have to choose a different set of quadratic functions.  

\bibliographystyle{alpha}
\bibliography{piecewiselyapunovpolicybib}

\newpage

\section*{Appendix}
\label{appendix}
In the appendix, we give details about the proofs of the propositions.
\begin{proposition}
Assume that Problem~\eqref{SDPlyap} has a feasible solution $(\Pq,\mathcal{W},\mathcal{U},\mathcal{Z},\alpha,\beta)$. Then:
\begin{enumerate}
\item The family $\Pq$ defines a PQL;
\item There exists $(\Pq,\mathcal{W},\mathcal{U},\mathcal{Z},\alpha,\beta)$ satisfiying~\eqref{boundedeqrelax},~\eqref{stabilityeqrelax}
and~\eqref{initialeqrelax} if and only if Problem~\eqref{SDPlyap} is feasible;
\item For all $(i,j)\in \overline{\Sw}$,
\[
\begin{array}{ll}
& {F^i}^\intercal\Mat(\Id,0,0)F^i\\
\preceq & \Mat(P^i,2q^i,-\alpha)+\Mat(0,0,\beta)\\
&-{E^{ij}}^\intercal\left( \begin{pmatrix}0_{n_i}  & 0_{n_i,n_j} \\
0_{n_j,n_i} & \pos{W^j}+\psd{W^j}\end{pmatrix}+\pos{U^{ij}}+\psd{U^{ij}}\right) E^{ij}\enspace ; 
\end{array}
\]
\item We have $\displaystyle{\sup_{x\in X^0} \norm{x}_2^2} \leq \beta$;
\item If $(\Pq,\mathcal{W},\mathcal{U},\mathcal{Z},\alpha,\beta)$ is optimal and $\alpha>0$ then $\displaystyle{\sup_{x\in X^0} L(x)}=\alpha$.
\end{enumerate}
\end{proposition}

\begin{proof}
(1) The first statement follows readily from Corollary~\eqref{eqcorollary}. 

(2) The "if" part is obvious. Let us focus on the "only if" part and let $(\Pq,\mathcal{W},\mathcal{U},\mathcal{Z},\alpha,\beta)$ satisfiying~\eqref{boundedeqrelax},~\eqref{stabilityeqrelax}. From Th.~\ref{boundedth}, $\beta\geq 0$. If $\alpha\geq 0$, the proof 
is finished. Hence, we suppose that $\alpha<0$ and let us prove that $(\Pq,\mathcal{W},\mathcal{U},\mathcal{Z},0,\beta-\alpha)$ is feasible for Problem~\eqref{SDPlyap}. First $\beta-\alpha\geq 0$ since $\beta\geq 0$ and $\alpha<0$. Second, 
$\Mat(P^i,2q^i,0)-\Mat(\Id,0,-(\beta-\alpha))-{E^i}^\intercal \left(\pos{W^i}+\psd{W^i}\right) E^i=
\Mat(P^i,2q^i,-\alpha)-\Mat(\Id,0,-\beta)-{E^i}^\intercal \left(\pos{W^i}+\psd{W^i}\right) E^i\succeq 0$
by the fact that $(\Pq,\mathcal{W},\mathcal{U},\mathcal{Z},\alpha,\beta)$ satisfies~\eqref{boundedeqrelax} and thus
$(\Pq,\mathcal{W},\mathcal{U},\mathcal{Z},0,\beta-\alpha)$ satisfies~\eqref{boundedeqrelax}. Since 
$\alpha$ and $\beta$ do not appear in~\eqref{stabilityeqrelax}, $(\Pq,\mathcal{W},\mathcal{U},\mathcal{Z},0,\beta-\alpha)$ satisfies~\eqref{stabilityeqrelax}. Finally,
\[
\begin{array}{cl}
 &-\Mat(P^i,2q^i,0)-{E^{i0}}^\intercal \left(\pos{Z^{0i}}+\psd{Z^{0i}}\right) E^{i0}\\
=&-\Mat(P^i,2q^i,\alpha-\alpha)-{E^{i0}}^\intercal \left(\pos{Z^{0i}}+\psd{Z^{0i}}\right) E^{i0}\\
=&\Mat(0,0,-\alpha)-\Mat(P^i,2q^i,-\alpha)-{E^{i0}}^\intercal \left(\pos{Z^{0i}}+\psd{Z^{0i}}\right) E^{i0}
\end{array}
\]
We conclude that $-\Mat(P^i,2q^i,0)-{E^{i0}}^\intercal \left(\pos{Z^{0i}}+\psd{Z^{0i}}\right) E^{i0}\succeq 0$ and 
thus $(\Pq,\mathcal{W},\mathcal{U},\mathcal{Z},0,\beta-\alpha)$ satisfies~\eqref{initialeqrelax}.

(3) Let $(i,j)\in\overline{\Sw}$. Since $j\in \ind$, 
\[
\Mat(P^j,2q^j,-\alpha)-\Mat(\Id,0,-\beta)-{E^j}^\intercal \left(\pos{W^j}+\psd{W^j}\right) E^j\succeq 0
\]
and thus 
\[
\begin{array}{r}
{F^i}^\intercal \left(\Mat(P^j,2q^j,-\alpha)-\Mat(\Id,0,-\beta)\right.\\
\left. -{E^j}^\intercal \left(\pos{W^j}+\psd{W^j}\right) E^j\right) F^i\succeq 0
\end{array}
\]
and 
\[
\begin{array}{l}
{F^i}^\intercal \Mat(P^j,2q^j,-\alpha)F^i-{F^i}^\intercal {E^j}^\intercal \left(\pos{W^j}+\psd{W^j}\right) E^j F^i\\
\succeq {F^i}^\intercal\Mat(\Id,0,-\beta)F^i
\end{array}
\]
Hence:
\[
\begin{array}{ll}
&{F^i}^\intercal\Mat(\Id,0,-\beta)F^i\\
\preceq& -{F^i}^\intercal {E^j}^\intercal \left(\pos{W^j}+\psd{W^j}\right) E^j F^i +\Mat(P^i,2q^i,0)\\
&-{E^{ij}}^\intercal \left(\pos{U^{ij}}+\psd{U^{ij}}\right) E^{ij}
\end{array}
\]
Note that ${F^i}^\intercal\Mat(0,0,-\beta)F^i=\Mat(0,0,-\beta)$ and thus: 
\[
\begin{array}{ll}
&{F^i}^\intercal\Mat(\Id,0,0)F^i\\
\preceq & -{F^i}^\intercal {E^j}^\intercal \left(\pos{W^j}+\psd{W^j}\right) E^j F^i +\Mat(P^i,2q^i,0)\\
&-{E^{ij}}^\intercal \left(\pos{U^{ij}}+\psd{U^{ij}}\right) E^{ij}+\Mat(0,0,\beta)
\end{array}
\]
We conclude by the definition of $E^{ij}$.

(4) Since $(\Pq,\mathcal{W},\mathcal{U},\mathcal{Z},\alpha,\beta)$ defines a PQL function, then the result of Th.~\ref{boundedth} holds that is 
$\rea\subseteq \{x\in\rd\mid \norm{x}_2^2\leq \beta\}$ and since $X^0\subseteq \rea$, $\sup_{x\in X^0} \norm{x}_2^2 \leq \beta$.

(5) Now assume that $(\Pq,\mathcal{W},\mathcal{U},\mathcal{Z},\alpha,\beta)$ is an optimal solution such that $\alpha>0$ and suppose that $\sup_{x\in X^0} L(x)\neq \alpha$. We remark that $\sup_{x\in X^0} L(x)=\sup_{i\in\In} \sup_{x\in X^i\cap X^0} L^i(x)$ and from Constraint~\eqref{initialeqrelax}, for all $i\in\In$, $X^i\cap X^0\subseteq \{x\mid L^i(x)\leq \alpha\}$. Hence for all $i\in\In$, $\sup_{x\in X^i\cap X^0} L^i(x)\leq \alpha$ and thus $\sup_{x\in X^0} L(x)\leq \alpha$. Let $\epsilon>0$ such that $\gamma=\alpha-\epsilon\geq 0$ and $\sup_{x\in X^0} L(x)\leq \gamma$.
Let us denote by $\mathbf{N}$ the matrix defined by $\mathbf{N}_{1,1}=1$ and $\mathbf{N}_{l,m}=0$ for all $(l,m)\in\{1,\ldots,d+1\}^2\backslash\{(1,1)\}$. We have $-\Mat(P^i,2q^i,-\gamma)-{E^{i0}}^\intercal \left(\pos{Z^{0i}}+\psd{Z^{0i}}\right) E^{i0}=
-\Mat(L^i)+\gamma N-{E^{i0}}^\intercal \left(\pos{Z^{0i}}+\psd{Z^{0i}}\right) E^{i0}=(\alpha-\epsilon)N-\Mat(L^i)-{E^{i0}}^\intercal \left(\pos{Z^{0i}}+\psd{Z^{0i}}\right) E^{i0}$. Let us remark since ${E_{1,1}^{i0}}$ is equal to 1, 
that ${E^{i0}}^\intercal N {E^{i0}}=N$. Thus, 
\[
\begin{array}{ll}
&-\Mat(P^i,2q^i,-\gamma)-{E^{i0}}^\intercal \left(\pos{Z^{0i}}+\psd{Z^{0i}}\right) E^{i0}\\
=&-\Mat(L^i)+\alpha N-{E^{i0}}^\intercal \left(\pos{Z^{0i}}+\epsilon N+\psd{Z^{0i}}\right) E^{i0}\enspace.
\end{array}
\] 
In a second time, 
\[
\begin{array}{ll}
&\Mat(P^i,2q^i,-\gamma)-\Mat(\Id,0,-\beta)-{E^i}^\intercal \left(\pos{W^i}+\psd{W^i}\right) E^i\\
=&\Mat(P^i,2q^i,-\alpha)-\Mat(\Id,0,-\beta)-{E^i}^\intercal \left(\pos{W^i}+\psd{W^i}\right) E^i\\
&+\epsilon N\enspace .
\end{array}
\]
From Constraint~\eqref{boundedeqrelax}, $\Mat(P^i,2q^i,-\gamma)-\Mat(\Id,0,-\beta)-{E^i}^\intercal \left(\pos{W^i}+\psd{W^i}\right) E^i$ is positive semidefinite. We conclude that $(\Pq,\mathcal{W},\mathcal{U},\mathcal{Z}',\gamma,\beta)$ with 
$\mathcal{Z}'=\{\left(\pos{Z^{i0}}+\epsilon N,\psd{Z^{i0}}\right)\in \Sp{n_{i0}}\times \Sdp{n_{i0}},i\in\In\}$ is feasible and 
$\gamma+\beta=\alpha+\beta-\epsilon$ thus  $(\Pq,\mathcal{W},\mathcal{U},\mathcal{Z},\alpha,\beta)$ cannot be optimal.
\end{proof}

\begin{proposition}
The following statements hold: 
\begin{enumerate}
\item $F(C(\betab))\subseteq C(\betab)\iff\sha{F}(\betab)\leq \betab$;
\item $\rea\subseteq\inf\{C(\betab)\mid \betab\in\rr^{d+1} \text{ s.t. }\sha{F}(\betab)\leq \betab\}$;
\item For all $l\in [d+1]$, $\sha{F_{ij,l}}(\betab)$ is the optimal value of quadratic program;
\item For all $k\in[d]$, $X_k^0=\displaystyle{\max\{(\inf_{x\in X^0} x_k)^2,(\sup_{x\in X^0} x_k)^2\}}$ and if $L$ is constructed from an optimal solution $(\Pq,\mathcal{W},\mathcal{U},\mathcal{Z},\alpha,\beta)$ of~\eqref{SDPlyap} such that $\alpha>0$, then $X_{d+1}^0=\alpha$.
\end{enumerate}
\end{proposition}
\begin{proof}
(1) $F(C(\betab))\subseteq C(\betab)$ iff for all $k\in[d]$, $\sup_{y\in F(C(\betab))} y_k^2\leq \betab_k$ and 
 $\sup_{y\in F(C(\betab))} L(y)\leq \betab_{d+1}$.
 Now for all $k\in[d]$:
 \[
 \begin{array}{ll}
 &\sup_{y\in F(C(\betab))} y_k^2\\
 =&\sup\{\sup_{y\in \map{}(C(\betab))} y_k^2,\sup_{y\in X^0} y_k^2\}\\
 =&\sup\{\sup_{(i,j)\in\overline{\Sw}}\displaystyle{\sup_{\substack{y=A^i x+b^i,\\ x\in C(\betab), x\in X^{ij}}} y_k^2},\sup_{y\in X^0} y_k^2\}\\
 =&\sha{F_k}(\betab)
 \end{array}
 \]
 and 
 \[
 \begin{array}{ll}
 &\sup_{y\in F(C(\betab))} L(y)\\
 =&\sup\{\sup_{y\in \map{}(C(\betab))} L(y),\sup_{y\in X^0} L(y)\}\\
 =&\sup\{\sup_{(i,j)\in\overline{\Sw}}\displaystyle{\sup_{\substack{y=A^i x+b^i,\\ x\in C(\betab),\ x\in X^{ij}}} L^i(y)},\sup_{y\in X^0} L(y)\}\\
 =&\sha{F_{d+1}}(\betab)
 \end{array}
 \]
 
 (2) From Eq.~\eqref{reatarski}, $\rea\subseteq \inf\{C(\betab)\mid \betab\in\rr^{d+1},\ \sha{F}(C(\betab))\subseteq C(\betab)\}$. We conclude using the first point.
 
 (3) Obvious.

(4) Let $k\in [d]$. Since $X^0$ is compact and $x\mapsto x_k$ is continuous then there exist $z\in X^0$ and $u\in X^0$ such that 
$z_k=\inf_{x\in X^0} x_k$ and $u_k=\sup_{x\in X^0} x_k$. Hence $z_k\leq x_k\leq u_k$ for all $x\in X^0$ and thus 
for all $x\in X^0$, $x_k^2\leq \max(z_k^2,u_k^2)$. Since $z$ and $u$ belong to $X^0$, then 
$X^0_k=\max(z_k^2,u_k^2)$. We have assumed that $(\Pq,\mathcal{W},\mathcal{U},\mathcal{Z},\alpha,\beta)$ is an optimal
solution of Problem~\eqref{SDPlyap} and $\alpha>0$ then ${X^0}^\dag(L)=\alpha$ from Prop.~\ref{PQLprop}.  
\end{proof}
\begin{proposition}[Safe overapproximation]
The following assertions are true:
\begin{enumerate}
\item For all $l\in[d+1]$, $\rel{F_l}$ is the optimal value of a SDP program;
\item $\sha{F}\leq \rel{F}$\enspace . 
 \end{enumerate}
\end{proposition}

\begin{proof}
(1) Obvious.

(2) We have to prove that for all $k\in [d+1]$, for all $\omega\in\rr^{d+1}$, $\sha{F_{ij,k}}(\betab)\leq \rel{F_{ij,k}}(\betab)$. We do the proof for the case $k=d+1$.
The other cases follows the same proof constructions.

Applying the weak duality theorem, we obtain:
\[
\begin{array}{r}
\sha{F_{ij,d+1}}(\betab)\leq\displaystyle{\inf_{\lambda\in\rr_+^{d+1}} \sup_{x\in X^{ij}} L^j(f^i(x))+\sum_{k=1}^{d}\lambda_k (\betab_k-x_k^2)}\\
+\lambda_{d+1}(\betab_{d+1}-L^i(x))
\end{array}
\]
Using Lemma~\ref{polyhedralcopo} and Corollary~\ref{copositiverestriction} we get:
\[
\begin{array}{ccl}
&\sha{F_{ij,d+1}}(\betab)\\
\leq &\displaystyle{\inf_{\lambda,\eta}} & \eta\\
&\st & \left\{
\begin{array}{l}
\forall\, x\in X^{ij},\\
\displaystyle{\eta-L^j(f^i(x))-\sum_{k=1}^{d}\lambda_k (\betab_k-x_k^2)}\\
\displaystyle{-\lambda_{d+1}(\betab_{d+1}-L^i(x))-p(f^i(x))\geq 0}\\
\\
\lambda\geq 0,\ \eta\in\rr
\end{array}
\right.\\
\leq &\displaystyle{\inf_{\lambda,\eta}}  & \eta\\
&\st & \left\{
\begin{array}{l}
\displaystyle{\Mat\left(\eta-L^j(f^i(x))-\sum_{k=1}^{d}\lambda_k (\betab_k-x_k^2)\right.}\\
\displaystyle{\left.-\lambda_{d+1}(\betab_{d+1}-L^i(x))\right)}\in \ccop{d+1}{E^{ij}}\\
\\
\lambda\in\rr_+^{d+1},\ \eta\in\rr
\end{array}
\right.\\
\leq & \displaystyle{\inf_{\lambda,\eta,Y,Z}} & \eta\\
&\st & \left\{
\begin{array}{l}
\displaystyle{\Mat\left(\eta-L^j(f^i(x))-\sum_{k=1}^{d}\lambda_k (\betab_k-x_k^2)\right.}\\
\displaystyle{\left.-\lambda_{d+1}(\betab_{d+1}-L^i(x))\right)}-{E^{ij}}^\intercal (Y+Z)E^{ij}\\
\succeq 0\\
\\
\lambda\in\rr_+^{d+1},\ \eta\in\rr,\ Y\geq 0,\ Z\succeq 0
\end{array}
\right.
\end{array}
\]

Now from Eq.~\eqref{lemmasimple} and since $A\to \Mat(A)$ is linear, we have:
\[
\begin{array}{l}
\displaystyle{\Mat\left(\eta-L^j(f^i(x))-\sum_{k=1}^{d}\lambda_k (\betab_k-x_k^2)-\lambda_{d+1}(\betab_{d+1}-L^i(x))\right)}\\
\displaystyle{=(\eta-\sum_{k=1}^{d+1} \lambda_k \betab_k)N- {F^i}^\intercal M_L^j F^i+\sum_{k=1}^{d} \lambda_k M_k+\lambda_{d+1} M_L^i}\\
=\displaystyle{(\eta-\sum_{k=1}^{d+1} \lambda_k \betab_k)N-\Phi_{ij,d+1}(\lambda,Y,Z)}+{E^{ij}}^\intercal (Y+Z)E^{ij}
\end{array}
\]
Finally:
$\Mat(\eta-L^j(f^i(x))-\sum_{k=1}^{d}\lambda_k (\betab_k-x_k^2)-\lambda_{d+1}(\betab_{d+1}-L^i(x)))-{E^{ij}}^\intercal (Y+Z)E^{ij}
=(\eta-\sum_{k=1}^{d+1} \lambda_k \betab_k)N-\Phi_{ij,d+1}(\lambda,Y,Z)$. Since $\rel{F_{ij,l}}$ is the infimum of $\eta$
over the constraint $(\eta-\sum_{k=1}^{d+1} \lambda_k \betab_k)N-\Phi_{ij,d+1}(\lambda,Y,Z)\succeq 0$, $\lambda\in\rr_+^{d+1},\ \eta\in\rr,\ Y\geq 0$ and $Z\succeq 0$,
this achieves the proof.
\end{proof}
\begin{proposition}
Let $(i,j)\in\overline{\Sw}$, $l\in [d+1]$, $\lambda\in\rr_+^{d+1}$.
The following statements are true:
\begin{enumerate}
\item $F_{ij,l}^{\lambda}$ is affine;
\item $F_{ij,l}^{\lambda}$, $\rel{F_{ij,l}}$ and $\rel{F_l}$ are monotone;
\item $\rel{F_{ij,l}}$ and $\rel{F_l}$ are upper semi-continuous. 
\end{enumerate}
\end{proposition}
\begin{proof}
The first assertion is straightforward from Equation~\eqref{affine}. The function $w\mapsto F_{ij,l}^{\lambda}(w)$
is monotone from the positivity of $\lambda$ and the two last functions comes are monotone as the supremum 
of monotone functions.  The function $w\mapsto \rel{F_{ij,l}}(w)$ is upper semi-continuous as the infimum of continuous functions and 
$w\mapsto \rel{F_{l}}(w)$ is upper semi-continuous as the finite supremum of upper semi-continuous functions. 
\end{proof}
\begin{proposition}
The vector $\betab^0$ satisfies $\rel{F}(\betab^0)\leq \betab^0$.
\end{proposition}
\begin{proof}
From Prop.~\ref{PQLprop}, we have for all $k\in [d]$, $X_k^0\leq \beta=\betab_k^0$ and $X_{d+1}^0=\alpha=\betab_{d+1}^0$.

Then it suffices to prove that for all $l\in [d+1]$, for all $(i,j)\in\overline{\Sw}$, $\rel{F_{ij,l}}(\betab^0)\leq \betab^0$. We can show it by proving that
 for all $l\in [d+1]$, for all $(i,j)\in\overline{\Sw}$, there exist $\lambda\geq 0$, $Y\geq 0$ and $Z\succeq 0$ such that:
\[
(\betab_l^0-\sum_{k=1}^{d+1} \lambda_k \betab_k)N-\Phi_{ij,l}(\lambda,Y,Z)\succeq 0
\]
Let us define $\bar{\lambda}$ by $\bar{\lambda}_{d+1}=1$ and $\bar{\lambda}_{k}=0$ for all $k\in [d]$. Let $(i,j)\in\overline{\Sw}$.

Recall that $(\Pq,\mathcal{W},\mathcal{U},\mathcal{Z},\alpha,\beta)$ is an optimal solution of Problem~\eqref{SDPlyap}.
Let $l=d+1$, and let us extract $\pos{U^{ij}}$ and $\psd{U^{ij}}$ from $\mathcal{U}$, then we have:
\[
\begin{array}{l}
(\betab_{d+1}^0-\sum_{k=1}^{d+1} \bar{\lambda}_k \betab_k)N-\Phi_{ij,l}(\bar{\lambda},\pos{U^{ij}},\psd{U^{ij}})\\
=-{F^i}^\intercal\Mat_L^j F^i+\Mat_L^i-{E^{ij}}^\intercal (\pos{U^{ij}}+\psd{U^{ij}})E^{ij}
\end{array}
\]
We conclude that $(\betab_{d+1}^0-\sum_{k=1}^{d+1} \bar{\lambda}_k \betab_k)N-\Phi_{ij,l}(\bar{\lambda},\pos{U^{ij}},\psd{U^{ij}})\succeq 0$
since $(\Pq,\mathcal{W},\mathcal{U},\mathcal{Z},\alpha,\beta)$ is an optimal solution of Problem~\eqref{SDPlyap} and thus 
satisfies~\eqref{stabilityeqrelax}. We conclude that $(\betab_{d+1}^0,\bar{\lambda},\pos{U^{ij}},\psd{U^{ij}})$ is a feasible solution 
of the SDP problem~\eqref{relaxedSDP} and thus $\rel{F_{ij,d+1}}(\betab^0)\leq \betab_{d+1}^0$.

Let $l\in[d]$, $\bar{Y}=\begin{psmallmatrix}0_{n_i}  & 0_{n_i,n_j} \\
0_{n_j,n_i} & \pos{W^j}\end{psmallmatrix}+\pos{U^{ij}}$ and 
$\bar{Z}=\begin{psmallmatrix}0_{n_i}  & 0_{n_i,n_j} \\
0_{n_j,n_i} & \psd{W^j}\end{psmallmatrix}+\psd{U^{ij}}$ where $\pos{W^j}$ and $\psd{W^j}$ are extracted from $\mathcal{W}$ and $\pos{U^{ij}}$ and $\psd{U^{ij}}$ are extracted from $\mathcal{U}$.
We have:
\[
\begin{array}{l}
(\betab_{l}^0-\sum_{k=1}^{d+1} \bar{\lambda}_k \betab_k)N-\Phi_{ij,l}(\bar{\lambda},\bar{Y},\bar{Z})\\
=\Mat(0,0,\beta-\alpha)-{F^i}^\intercal\Mat_l F^i+\Mat_L^i-{E^{ij}}^\intercal (\bar{Y}+\bar{Z})E^{ij}
\end{array}
\]
Now, remark that $\Mat_l\preceq \Mat(\Id,0,0)$ and thus 
$-{F^i}^\intercal\Mat_l F^i+\Mat(P^i,2q^i,-\alpha)-{E^{ij}}^\intercal (\bar{Y}+\bar{Z})E^{ij}+\Mat(0,0,\beta)
\preceq -{F^i}^\intercal\Mat(\Id,0,0)F^i+\Mat(P^i,2q^i,-\alpha)-{E^{ij}}^\intercal (\bar{Y}+\bar{Z})E^{ij}+\Mat(0,0,\beta)$.
The right-hand-side sum of matrices is positive semi-definite from the second assertion of Prop.~\ref{PQLprop}.
We conclude that $(\betab_l^0(p),\bar{\lambda},\bar{Y},\bar{Z})$ is a feasible solution 
of the SDP problem~\eqref{relaxedSDP} and thus $\rel{F_{ij,l}}(\betab^0)\leq \betab_l^0$.
\end{proof}

\end{document}